\documentclass[12pt]{amsart}
\usepackage{tikz-cd} 
\usepackage{xy}
\usepackage{graphicx}
\usepackage{subfigure}
\usepackage{amsmath,amssymb,amsthm,amsfonts,amsrefs}
\usepackage{newclude}
\usepackage{multirow}
\usepackage{dynkin-diagrams}
\counterwithin{equation}{section}
\newtheorem{lm}{Lemma}[section]
\newtheorem{theorem}[lm]{Theorem}
\newtheorem{cor}[lm]{Corollary}
\newtheorem{prop}[lm]{Proposition}

\theoremstyle{definition}
\newtheorem{defn}[lm]{Definition}

\theoremstyle{remark}
\usepackage{latexsym}
\usepackage{CJKutf8}
\usepackage{indentfirst}
\usepackage{mathrsfs}
\usepackage{geometry}
\usepackage{imakeidx}
\usepackage{appendix}
\usepackage[colorlinks=true,hyperfootnotes=true]{hyperref}
\usepackage{mathdots}
\usepackage{arydshln}

\newcommand{\BC}{{\mathbb {C}}}

\newcommand{\CN}{{\mathcal {N}}}

\newcommand{\RO}{{\mathrm {O}}}

\newcommand{\End}{{\mathrm{End}}}

\newcommand{\GL}{{\mathrm{GL}}}

\newcommand{\Hom}{{\mathrm{Hom}}}

\newcommand{\Mat}{{\mathrm{Mat}}}

\newcommand{\rn}{\mathrm{n}}

\newcommand{\rs}{\mathrm{s}}

\newcommand{\Spec}{{\mathrm{Spec}}}

\newcommand{\Sp}{{\mathrm{Sp}}}

\newcommand{\tr}{{\mathrm{tr}}}

\newcommand{\triv}{\mathrm{triv}}

\newcommand{\diag}{\operatorname{diag}}
\newcommand{\sgn}{\operatorname{sgn}}

\newcommand{\Span}{\operatorname{Span}}

\newcommand{\g}{\mathfrak g}
\newcommand{\gC}{{\mathfrak g}_{\C}}

\newcommand{\h}{\mathfrak h}

\renewcommand{\u}{\mathfrak u}

\newcommand{\gl}{\mathfrak g \mathfrak l}
\renewcommand{\sp}{\mathfrak s \mathfrak p}
\newcommand{\fM}{\mathfrak M}
\newcommand{\fN}{\mathfrak N}

\newcommand{\fX}{\mathfrak X}

\newcommand{\Z}{\mathbb{Z}}
\newcommand{\C}{\mathbb{F}}
\newcommand{\R}{\mathbb R}

\newcommand{\B}{\mathbf{B}}
\newcommand{\G}{\mathbf{G}}

\newcommand{\Y}{\mathbf{Y}}

\newcommand{\be}{\mathbf{e}}

\newcommand{\bac}{\backslash\!\!\backslash}
\renewcommand{\mid}{\,:\,}

\newcommand{\cf}{\emph{cf.}~}

\title[Closed orbits and descents]{Closed Orbits and Descents for Enhanced Standard Representations of Classical Groups}
\date{}
\author{Chen Liang}
\address{School of Mathematical Sciences, Zhejiang University\\
	Hangzhou, 310058, China}\email{12235009@zju.edu.cn}

\begin{document}
	\maketitle
	\begin{abstract} 
    Let $G=\GL_n(\C)$, $\RO_n(\C)$, or $\Sp_{2n}(\C)$ be one of the classical groups over an algebraically closed field $\C$ of characteristic $0$, let $\breve{G}$ be  the MVW-extension  of $G$,  and let $\g$ be the Lie algebra of $G$. 
    In this paper, we classify the closed orbits in the enhanced standard representation $\g\times E$ of $G$, 
    where $E$ is the natural representation if $G=\RO_n(\C)$ or $\Sp_{2n}(\C)$, and is the direct sum of the natural representation and its dual if $G=\GL_n(\C)$.
 Additionally, for every closed $G$-orbit in $\g\times E$, we prove that it is $\breve{G}$-stable, and determine explicitly the corresponding stabilizer group as well as the action on the normal space. 
	\end{abstract}
	\tableofcontents
	\section{Introduction and main results}
    Throughout this paper, let $\C$ be an algebraically closed field of characteristic 0.
    \subsection{Motivations}
Let $\G$ be a reductive algebraic group over a local field $F$ of characteristic $0$.  Function spaces for rational representations of $\G$ play an important role in the study of representations of $\G(F)$. One classical problem is the ``multiplicity one problem'': Given a larger group $\G'(F)$ and an irreducible smooth admissible representation $\pi$ of $\G'(F)$, is the dimension of $\Hom_{\G(F)}(\pi,\BC)$ at most $1$? By the Gelfand-Kazhdan criterion \cite{GK,SZ1} and the Harish-Chandra descent \cite{AG}, the multiplicity one problem can be reduced to proving that certain $\breve{\G}(F)$-equivariant generalized functions on a rational representation of $\G(F)$ must be $0$, where $\breve{\G}(F)$ is an extension of $\G(F)$ by $\{\pm 1\}$. Another problem is the existence of smooth transfer. Let $V$ be a rational representation of $\G$. Given another reductive group $\G'$ over $F$ and a rational representation $V'$ of $\G'$, suppose that there exists a matching of regular semismiple orbits between $V$ and $V'$. The smooth transfer of a Schwartz function $f$ on $V(F)$ is a Schwartz function $f'$ on $V'(F)$ such that $O_{\gamma}(f)=\Delta(\gamma,\gamma')O_{\gamma'}(f')$, whenever a regular semisimple $\gamma\in V(F)$ matches a regular semisimple $\gamma'\in V'(F)$, where $O_{\gamma}(f),O_{\gamma'}(f')$ are suitably defined orbit integrals, and $\Delta(\gamma,\gamma')$ is the transfer factor.

To study the function spaces of the rational representation $V(F)$, one needs to study the geometry of the action of $\G$ on $V$ first, in particular, the classification of closed orbits, the corresponding stabilizer group and the descendants (see Definition \ref{defn:descent}). In this article, we investigate these geometric properties for the enhanced standard representation of classical groups as well as their MVW-extensions (see \S\,\ref{sec:enrepn} for the definitions). Our results may be applied in the proof of multiplicity one theorem and the existence of smooth transfer. In the proof of multiplicity one theorem \cite{AGb,AGRS,SZ}, as pointed out in \cite{AGb}, by applying \cite{AG}*{Theorem 3.2.1}, one of our results (Theorem \ref{thm:mainmvwsimple}) can provide a direct proof of \cite{AGb}*{Proposition 3.2.1}, \cite{AGRS}*{Propositions 3.2, 5.2} and \cite{SZ}*{Propositions 7.1 and 7.2}. In the proof of the existence of smooth transfer \cite{Zha14, Xue19, CZ21} for Jacquet-Rallis relative trace formulas, the action on the general linear side is reduced to the enhanced standard representation of $\GL_n$, and its descendants are used for further reduction.

	\subsection{Enhanced standard representations} \label{sec:enrepn}
 Let $G$ be one of the following classical groups:
	\begin{equation}
    \label{eq:classicalgroups}
		\GL_n(\C),\quad\RO_n(\C)\quad\text{and}\quad\Sp_{2n}(\C)\qquad (n\ge 0),
	\end{equation}
	where $\GL_n(\C)$ is the general linear group of rank $n$ over $\C$, 
	\begin{equation*}
		\RO_n(\C)=\left\{g\in\GL_n(\C)\mid g^t\alpha_ng=\alpha_n\right\}
	\end{equation*}
    is the orthogonal group, 
	and
	\begin{equation*}
		\Sp_{2n}(\C)=\left\{g\in\GL_{2n}(\C)\mid g^t\beta_{2n}g=\beta_{2n}\right\}
	\end{equation*}
    is the symplectic group. Here,
    \begin{equation*}
        \alpha_{n}=\left(\begin{matrix}
				&I_{\frac{n}{2}}\\I_{\frac{n}{2}}&
			\end{matrix}\right),\text{ if n even},\quad\alpha_n=
            \left(\begin{matrix}
				1&&\\&&I_{\frac{n-1}{2}}\\&I_{\frac{n-1}{2}}&
			\end{matrix}\right),\text{ if n is odd},
    \end{equation*}
    and
    \begin{equation*}
        \beta_{2n}=\left(\begin{matrix}
			&I_n\\-I_n&
		\end{matrix}\right).
    \end{equation*}
    
    Write $\breve{G}=G\rtimes\{\pm 1\}$ for the \emph{MVW-extension} of $G$ (cf.\,\cite{MVW},\cite{Sun}), where 
    \begin{equation*}
   (-1).g   = \begin{cases*} 
   g^{-t}, & if $G=\GL_n(\C)$,\\
   \alpha_{n}g\alpha_{n},& if $G=\RO_n(\C)$,\\
\alpha_{2n}g\alpha_{2n}, & if $G=\Sp_{2n}(\C)$. 
        \end{cases*}
    \end{equation*}
    In this paper, we view $\breve{G}$ and $G$ as algebraic groups over $\C$. 

	Let $\g=\mathfrak{gl}_n(\C)$, $\mathfrak{o}_n(\C)$, and $\mathfrak{sp}_{2n}(\C)$  be the Lie algebras of $G=\GL_n(\C)$, $\RO_n(\C)$, and $\Sp_{2n}(\C)$, respectively. Define an action of $\breve{G}$ on $\g$ by letting 
    \begin{equation*}
    (g,1).X=gXg^{-1}\quad \text{and}\quad    (g,-1).X=\begin{cases*}
            gX^tg^{-1},&if $G=\GL_n(\C)$,\\
            -g\alpha_{n}X\alpha_{n}g^{-1},&if $G=\RO_n(\C)$,\\
            -g\alpha_{2n}X\alpha_{2n}g^{-1},&if $G=\Sp_{2n}(\C)$.
        \end{cases*}
    \end{equation*}
   Additionally, we define an action of $\breve{G}$ on  the space 
	\begin{equation*}
		E:=\begin{cases*}
			\C^{n\times 1}\times \C^{1\times n},&if $G=\GL_n(\C)$,\\
			\C^{n\times 1},&if $G=\RO_n(\C)$,\\
            \C^{2n\times 1}, &if $G=\Sp_{2n}(\C)$
		\end{cases*}
	\end{equation*}
as follows: 

\begin{equation*}
    (g,\delta).(u,v)=\begin{cases*}
        (gu,vg^{-1}),&if $G=\GL_n(\C)$ and $\delta=1$,\\
        (-gv^t,-u^tg^{-1}),&if $G=\GL_n(\C)$ and $\delta=-1$,
    \end{cases*}
\end{equation*}
and
\begin{equation*}
    (g,\delta).u=\begin{cases*}
        gu,&if $G=\Sp_{2n}(\C)$ or $\RO_n(\C)$ and $\delta=1$,\\
        -g\alpha_{2n}u,&if $G=\Sp_{2n}(\C)$ and $\delta=-1$,\\
        -g\alpha_nu,&if $G=\RO_n(\C)$ and $\delta=-1$
    \end{cases*}
\end{equation*}
    Here, $\C^{p\times q}$ ($p,q\ge 1$) denotes the space of $p\times q$-matrices over $\C$. 
     
	Let $\breve{G}$ act on $\g\times E$ diagonally, so that it is a rational representation of $\breve{G}$ (and $G$ by taking restriction). 
    We call this representation of $\breve{G}$ (resp. $G$) the \emph{enhanced standard representation} of $\breve{G}$ (resp. $G$). 
 The main goal of this paper is to classify the closed $G$-orbits and $\breve{G}$-orbits in  $\g\times E$, and determine their corresponding stabilizer subgroups.

  \subsection{Related works}
     Denote by $\CN_{\g}$ the null cone of $\g$, which consists of all nilpotent matrices in $\g$. The closed $G$-orbits in $\g$ as well as the $G$-orbits in $\CN_{\g}$ have been completely classified (see \cite{CM} for example). It is known that every  closed $G$-orbit in $\g$ is $\check{G}$-stable (see \cite{MVW}). In \cite{Kato}, Kato considered an exotic nilpotent cone and derived the Deligne–Langlands theory for those exotic nilpotent orbits. To compute the local intersection cohomology of orbit closures in the exotic nilpotent cone, Achar and Henderson studied in \cite{AH08} the so-called ``enhanced nilpotent cone'' $\CN_{\gl_n(\C)}\times\C^{n\times 1}$ and classified its $\GL_n(\C)$-orbits.

  On the other hand, the  $G$-orbits in $\g\times E$ or $\CN_\g\times E$ have been studied in literature for various motivations. 
In \cite{RS}, for the purpose of proving the multiplicity one conjectures, Rallis and Schiffmann  studied the enhanced standard representation $\g\times E$ of $G$, and gave a criterion for a $G$-orbit in $\g\times E$ to be closed. In \cite{NO}, to generalize the results of \cite{Kato,AH08}, Nishiyama and Ohta determined regular semisimple $\GL_n(\C)$-orbits and the structure of the null cone in $\gl_n(\C)\times\left(\C^{n\times 1}\right)^p\times\left(\C^{1\times n}\right)^q$. In order to  generalize Ohta's conditions in \cite{Oht}, Nishiyama gave in \cite{Nis14} certain sufficient conditions for the map between orbit spaces induced by the inclusions of algebraic groups and varieties to be injective, and showed that the natural embedding of $\sp_{2n}(\C)\times\C^{2n\times 1}\hookrightarrow\gl_{2n}(\C)\times\C^{2n\times 1}\times\C^{1\times 2n}$ induces an injection of orbit spaces.

In what follows, we outline the main results of this paper. 
 
	\subsection{Closed orbits}
	\label{sec:closed}
	For $g_1\in\gl_{m_1}(\C),\dots,g_b\in\gl_{m_b}(\C)$, denote by $\diag(g_1,\dots,g_b)$ the matrix
	\begin{equation*}
		\begin{matrix}
			\left(\begin{matrix}
				g_1&&\\&\ddots&\\&&g_b
			\end{matrix}\right)\in\gl_{m_1+\dots+m_b}(\C).
		\end{matrix}
	\end{equation*}
	
	 To classify the closed $G$-orbits in $\g\times E$, we first study the closed $G$-orbits in $\CN_{\g}\times E$. For this purpose, we construct a subset $\fN_G$ of $\CN_{\g}\times E$ as follows:

	If $G=\GL_n(\C)$, then $\fN_G$ consists of all elements in $\CN_\g\times E$ which have the  form
		\begin{equation}
			\label{eq:closedorbitnullconeGL}
			x(k,y_1,\dots,y_k)=(\diag(J_k,0_{(n-k)\times(n-k)}),(y_1\dots,y_k,0\dots,0)^t,(1,0,\dots,0)),
		\end{equation}
where $k\geq 0$,  $y_1,\dots,y_k\in\C$ with $y_k\neq 0$, and \begin{equation*}
	J_0=0\quad\text{and}\quad	J_k=\left(\begin{matrix}
			0&1&&\\&\ddots&\ddots&\\&&0&1\\&&&0
		\end{matrix}\right)\ (k\ge 1).
	\end{equation*}
        
    If $G=\Sp_{2n}(\C)$, then $\fN_G$ consisting of all pairs $x=(X,u)\in\CN_\g\times E$ such that
		\begin{equation}
			\label{eq:nilpotentclosedspm}
			X=\left(\begin{matrix}
				\begin{array}{c:c}
					\diag(J_k,0_{(n-k)\times(n-k)})&\be_{kk}(n)\\\hdashline
					0&\diag(-J_k^t,0_{(n-k)\times(n-k)})
				\end{array}
			\end{matrix}\right)
		\end{equation}
		and
		\begin{equation}
			\label{eq:nilpotentclosedspv}
			u=(\underbrace{u_1,\dots,u_k}_k,\underbrace{0,\dots,0}_{n-k},\underbrace{u_{n+1},\dots,u_{n+k}}_k,\underbrace{0,\dots,0}_{n-k})^t,
		\end{equation}
        where $k\geq 0$, $u_1,\dots,u_k,u_{n+1},\dots,u_{n+k}\in\C$ with $u_{n+1}\neq 0$, and $\be_{i,j}(n)$ denotes the $n\times n$-matrix whose $(i,j)$-entry is 1 and other entries are zero. For $x\in\fN_G$ above, define $r(x)=k$.
        
        If $G=\RO_{2n+1}(\C)$, then $\fN_G$ consists of all pairs $(X,u)\in\CN_\g\times E$ such that
		\begin{equation}
			\label{eq:nilpotentclosedooddm}
			X=\left(\begin{matrix}
				\begin{array}{c:c:c:c:c}
					0&\begin{matrix}
						0&\dots&0
					\end{matrix}&\begin{matrix}
						0&\dots&0
					\end{matrix}&\begin{matrix}
						0&\ldots&0&1
					\end{matrix}&\begin{matrix}
						0&\dots&0
					\end{matrix}\\\hdashline
					\begin{matrix}
						0\\\vdots\\0\\-1
					\end{matrix}&J_k&0&0&0\\\hdashline
					\begin{matrix}
						0\\\vdots\\0
					\end{matrix}&0&0_{(n-k)\times(n-k)}&0&0\\\hdashline
					\begin{matrix}
						0\\\vdots\\0
					\end{matrix}&0&0&-J_k^t&0\\\hdashline
					\begin{matrix}
						0\\\vdots\\0
					\end{matrix}&0&0&0&0_{(n-k)\times(n-k)}
				\end{array}
			\end{matrix}\right)
		\end{equation}
		and
		\begin{equation}
			\label{eq:nilpotentclosedooddv}
		    u=(u_1,\underbrace{u_2,\dots,u_{k+1}}_k,\underbrace{0,\dots,0}_{n-k},\underbrace{u_{n+2},\dots,u_{n+k+1}}_k,\underbrace{0,\dots,0}_{n-k})^t,
		\end{equation}
        where $k\ge 0$, and $u_1,\dots,u_{k+1},u_{n+2},\dots,u_{n+k+1}\in\C$ with $u_{n+2}\neq 0$. For $x\in\fN_G$ above, define $r(x)=k$.
        
        If $G=\RO_{2n}(\C)$, then $\fN_G$ consists of all pairs $(X,u)\in\CN_\g\times E$ such that
		\begin{equation}
			\label{eq:nilpotentclosedoevenm}
			X=\left(\begin{matrix}
				\begin{array}{c:c}
					\diag(J_k,0_{(n-k)\times(n-k)})&\be_{k-1,k}(n)-\be_{k,k-1}(n)\\\hdashline
					0&\diag(-J_k^t,0_{(n-k)\times(n-k)})
				\end{array}
			\end{matrix}\right)
		\end{equation}
		and
		\begin{equation}
			\label{eq:nilpotentclosedoevenv}
			u=(\underbrace{u_1,\dots,u_k}_k,\underbrace{0,\dots,0}_{n-k},\underbrace{u_{n+1},\dots,u_{n+k}}_k,\underbrace{0,\dots,0}_{n-k})^t,
		\end{equation}
        where $k\ge 0$, and $u_1,\dots,u_k,u_{n+1},\dots,u_{n+k}$ with $u_{n+1}\neq 0$. For $x\in\fN_G$ above, define $r(x)=k$.

        The following theorem characterizes the closed $G$-orbits in $\CN_\g\times E$ whose proof is given in Section \ref{sec:4}.
	\begin{theorem}
		\label{thm:closedorbitsnullcone}
		For every $x\in\fN_G$, $Gx$ is a closed $G$-orbit in $\CN_\g\times E$. Conversely, each closed $G$-orbit in $\CN_{\g}\times E$ has such a form.
	\end{theorem}
    
For every $x,x'\in\fN_{G}$, we also give a sufficient and necessary condition for $Gx=Gx'$ (see Propositions \ref{prop:closedorbitglndifferent} and \ref{prop:closeddifferent}). For example, if $G=\GL_n(\C)$, then $Gx=Gx'$ if and only if $k=k'$ and $y_1=z_1,\dots,y_k=z_k$, where $x=x(k,y_1,\dots,y_k)$ and $x'=x(k',z_1,\dots,z_{k'})$ are as in \eqref{eq:closedorbitnullconeGL}.
    
Now we are going to describe the closed $G$-orbits in $\g\times E$. To this end, we construct a subset $\fX_G$ of $\g\times E$ as follows:

    If $G=\GL_n(\C)$, then $\fX_G$ consists of all elements in $\g\times E$ which have the form
    \begin{equation}
			\label{eq:closedorbitnilpotentgln}
			(\diag(c_1I_{n_1}+N_1,\dots,c_bI_{n_b}+N_b),(\left(u^{(1)}\right)^t,\dots,\left(u^{(b)}\right)^t)^t,(v^{(1)},\dots,v^{(b)})),
		\end{equation}
        where $n_1,\dots,n_b\geq 1$ such that $n_1+\dots+n_b=n$, $c_i\neq c_j$ for $1\leq i\neq j\leq b$, and $(N_i,u^{(i)},v^{(i)})\in\fN_{\GL_{n_i}(\C)}$.

    If $G=\Sp_{2n}(\C),\RO_{2n+1}(\C)$ or $\RO_{2n}(\C)$, then $\fX_G$ consists of all pairs $(X,u)\in\g\times E$ satisfying the following conditions:
    
    $\bullet$
    \begin{equation}
        \label{eq:closedsom1}
        X=\left(\begin{matrix}
			\begin{array}{c:c:c:c}
				N_0^{(1)}&&N_0^{(2)}&\\\hdashline
				&\diag(c_1I_{n_1}+N_1,\dots,c_bI_{n_b}+N_b)&&\\\hdashline
				&&N_0^{(3)}&\\\hdashline
				&&&\diag(-c_1I_{n_1}-N_1^t,\dots,-c_bI_{n_b}-N_b^t)
			\end{array}
		\end{matrix}\right),
     \end{equation}
            where $c_i\neq\pm c_j$ for $1\leq i\neq j\leq b$, $N_0^{(1)}\in\C^{n_0^{\prime}\times n_0^{\prime}}$, $N_0^{(2)}\in\C^{n_0^{\prime}\times n_0}$, $N_0^{(3)}\in\C^{n_0\times n_0}$, and $N_i\in\CN_{\gl_{n_i}(\C)}$ for $1\leq i\leq b$, with $n_0\geq 0$ and $n_0^{\prime},n_1,\dots,n_b\geq 1$ such that
		\begin{equation*}
			n_0+n_1+\dots+n_b=n\quad\text{and}\quad n_0^{\prime}=\begin{cases*}
				n_0+1,&if $G=\RO_{2n+1}(\C)$,\\
				n_0,&if $G=\Sp_{2n}(\C)$ or $\RO_{2n}(\C)$;
			\end{cases*}
		\end{equation*}
        
        $\bullet$
        \begin{equation}
        \label{eq:closedsov}
            u=\left(\left(u^{(0)}\right)^t,\left(u^{(1)}\right)^t,\dots,\left(u^{(b)}\right)^t,\left(v^{(0)}\right)^t,\left(v^{(1)}\right)^t,\dots,\left(v^{(b)}\right)^t\right)^t
        \end{equation}
        such that, for $1\leq i\leq b$, $(N_i,u^{(i)},\left(v^{(i)}\right)^t)\in\fN_{\GL_{n_i}(\C)}$, and
        \begin{equation*}
            \left(\left(\begin{matrix}
				N_0^{(1)}&N_0^{(2)}\\&N_0^{(3)}
			\end{matrix}\right),\left(\begin{matrix}
			    u^{(0)}\\v^{(0)}
			\end{matrix}\right)\right)\in\fN_{G_0},
        \end{equation*}
        where
        \begin{equation}
        	\label{eq:subgroup}
            G_0=\begin{cases*}
                \Sp_{2n_0}(\C),&if $G=\Sp_{2n}(\C)$,\\
                \RO_{2n_0+1}(\C),&if $G=\RO_{2n+1}(\C)$,\\
                \RO_{2n_0}(\C),&if $G=\RO_{2n}(\C)$.
            \end{cases*}
        \end{equation}

        Based on Theorem \ref{thm:closedorbitsnullcone}, we prove in Section \ref{sec:5} the following classification result. 
        
	\begin{theorem}
		\label{thm:closedorbit}
		For every $x\in\fX_G$, $Gx$ is a closed $G$-orbit in $\g\times E$. Conversely, each closed $G$-orbit in $\g\times E$ has such a form.
	\end{theorem}
    
In Section \ref{sec:5}, we also prove     
the following result, which says that the closed orbits in $\g\times E$ of  $\breve{G}$ coincide with that of $G$. 

    \begin{theorem}
		\label{thm:closedorbitsmvw}
        Every $G$-closed orbit in $\g\times E$ is $\breve{G}$-stable.
	\end{theorem}
	\subsection{Descents}
    Let $O$ be a closed $\breve{G}$-orbit in $\g\times E$, and let $x\in O$. We denote by $\breve{G}_x$ the stabilizer of $\breve{G}$ at $x$, denote by $N_O^{\g\times E}$ the normal bundle of $O$ in $\g\times E$, and  denote by $N_{O,x}^{\g\times E}$ the fiber of $N_O^{\g\times E}$ at $x$.
    
	\begin{defn}
		\label{defn:descent}
		The natural action $\breve{G}_x$ on $N_{O,x}^{\g\times E}$ is called the \emph{descendant} of the enhanced standard representation at $x$.
	\end{defn}
%
%
%

To describe such descendants, we need to define the MVW extension for a product of classical groups as well as its enhanced standard representation. 
Let $H_1,\dots,H_r$ be classical groups as in \eqref{eq:classicalgroups}, and set $H=H_1\times\dots\times H_r$. 
For $i=1,\dots,r$, write $\h_i\times E_i$ for the enhanced standard representation of $\breve{H}_i$ .

	\begin{defn}
		\label{defn:fiberproduct}
		We define the \emph{MVW extensions} $\breve{H}$ of $H$ to be the fiber product
		\begin{equation*}
			\breve{H}_1\times_{\{\pm 1\}}\dots\times_{\{\pm 1\}}\breve{H}_r:=\{(h_1,\dots,h_r,\delta)\mid (h_1,\delta)\in\breve{H}_1,\dots,(h_r,\delta)\in\breve{H}_r\}.
		\end{equation*}
       Additionally, we call the natural representation
        \begin{equation*}
            \h^{\mathrm{en}}=(\h_1\times E_1)\times\dots\times(\h_r\times E_r)
        \end{equation*}
         of $\breve{H}$ the  \emph{enhanced standard representation} of $\breve{H}$.
	\end{defn}
    
 We denote by $\triv$ the trivial representation of $\breve{H}$. In the following result, we determine the descendants of enhanced standard representations $\g\times E$ of $\check{G}$. 
    
	\begin{theorem}
		\label{thm:mainmvwsimple}
		Let $O$ be a closed $\breve{G}$-orbit in $\g\times E$ and let $x\in O$. 
		
		\noindent(1) If $G=\GL_n(\C)$, then there exist $k\ge 0$ and $k_1,\dots,k_b\geq 1$ such that
		\begin{equation*}
		k+k_1+\dots+k_b=n,\quad	\breve{G}_x\simeq\breve{H},\quad \text{and}\quad 
            N_{O,x}^{\g\times E}\simeq \h^{\mathrm{en}}\oplus\triv^{2k},
		\end{equation*}
        where $H=\GL_{k_1}(\C)\times\dots\times\GL_{k_b}(\C)$.
		
		\noindent(2) If $G=\Sp_{2n}(\C)$, then there exist $l,k\geq 0$ and $k_1,\dots,k_b\geq 1$ such that
		\begin{equation*}
			k+l+k_1+\dots+k_b=n,\quad\breve{G}_x\simeq\breve{H}\quad \text{and}\quad 
            N_{O,x}^{\g\times E}\simeq \h^{\mathrm{en}}\oplus\triv^{2k},
		\end{equation*}
		where $H=\Sp_{2l}(\C)\times\GL_{k_1}(\C)\times\dots\times\GL_{k_b}(\C)$.
		
		\noindent(3) If $G=\RO_{n}(\C)$, then there exist $\gamma\in\{0,1\}$, $k,l\geq 0$ and $k_1,\dots,k_b\geq 1$ such that
		\begin{equation*}
			k+l+2k_1+\dots+2k_b+\gamma=n,\quad \breve{G}_x\simeq\breve{H}\quad \text{and}\quad 
            N_{O,x}^{\g\times E}\simeq \h^{\mathrm{en}}\oplus\triv^{2k+\gamma},
		\end{equation*}
		where $H=\RO_{l}(\C)\times\GL_{k_1}(\C)\times\dots\times\GL_{k_b}(\C)$.
	\end{theorem}
	
    The proof of Theorem \ref{thm:mainmvwsimple} is given in Section \ref{sec:6}.
	\section{Preliminaries}
    
\subsection{General notation}
\begin{itemize}
    \item  In this paper, all the (algebraic) varieties and 
groups are defined over $\C$.
\item We consider finite-dimensional vector spaces over $\C$ as algebraic varieties.
\item For the vector space $\C^{n\times 1}$, denote by $\{\be_1,\dots,\be_n\}$ its standard basis.
\item For an algebraic group $H$ acting on a variety $X$, a point $x\in X$, and a subset $K$ in $H$, we denote by
\begin{itemize}
\item  $X^H$ the set of all points in $X$ fixed by $H$,
\item $(H\bac X,\pi)$ the categorical quotient of $X$ by $H$ (if it exists), 
\item $H_x$ the stabilizer of $x$,
\item $Hx$ the $H$-orbit of $x$ in $X$, and
\item  $Z_H(K)$ the centralizer of $K$ in $H$.
\end{itemize}

\item For a Lie algebra $\h$ acting on a vector space $V$ and a vector $x\in V$, denote by $\h_x$ the stabilizer of $x$ in $\h$, and by $\h x$ the $\h$-orbit of $x$ in $V$.

\item For a variety $X$ and a point $x\in X$, denote by $T_xX$ the tangent space of $X$ at $x$. For a subvariety $Y$ of $X$ containing $x$, denote by $N_{Y,x}^X=(T_xX|_Y)/T_xY$ the normal space of $Y$ in $X$ at $x$.

\item For $y=(y_1,\dots,y_n)\in\C^{n\times 1}$, put
\begin{equation*}
	\phi(y)=\left(\begin{matrix}
		y_1&y_2&\dots&y_n\\
		&y_1&\dots&y_{n-1}\\
		&&\ddots&\vdots\\
		&&&y_1
	\end{matrix}\right)\quad\text{and}\quad\psi(y)=\left(\begin{matrix}
	    y_1&\dots&y_{n-1}&y_n\\
	    y_2&\dots&y_{n}&\\
	    \vdots&\iddots&&\\
	    y_n&&&
	\end{matrix}\right).
\end{equation*}
\end{itemize}
	\subsection{MVW extensions}
	Let $V$ be a vector space over $\C$. Denote by $\GL(V)$ the group of $\C$-linear automorphisms of $V$, and by $\gl(V)$ the Lie algebra of $\GL(V)$. The MVW-extension of $\GL(V)$ is
	\begin{equation*}
		\breve{\GL}(V)=\GL(V)\rtimes\{\pm 1\},
	\end{equation*}
	where $-1$ acts on $\GL(V)$ by transposition. 
      By specifying a basis of $V$, we have $\GL(V)=\GL_n(\C)$ and $\breve{\GL}(V)=\breve{\GL}_n(\C)$, where $n=\dim_{\C}V$.
	
	Assume now that $V$ is equipped with a quadratic or symplectic form $\langle\cdot,\cdot\rangle$. Put
	\begin{equation*}
		G(V)=\{g\in\GL(V)\mid\langle gv,gw\rangle=\langle v,w\rangle\text{ for }v,w\in V\}.
	\end{equation*}
	The Lie algebra of $G(V)$ is
	\begin{equation*}
		\g(V)=\{X\in\gl(V)\mid\langle Xv,w\rangle+\langle v,Xw\rangle=0\text{ for }v,w\in V\}.
	\end{equation*}
	Additionally, the MVW-extension $\breve{G}(V)$ of $G(V)$ is the subgroup of $\GL(V)\times\{\pm 1\}$ consisting of pairs $(g,\delta)$ such that either
	\begin{equation*}
		\delta=1\quad\text{and}\quad\langle gv,gw\rangle=\langle v,w\rangle\quad\text{for }v,w\in V,
	\end{equation*}
	or
	\begin{equation*}
		\delta=-1\quad\text{and}\quad\langle gv,gw\rangle=\langle w,v\rangle\quad\text{for }v,w\in V.
	\end{equation*}
    Let $\breve{G}(V)$ act on $\g(V)\times V$ by
    \begin{equation*}
        (g,\delta).(X,u)=(\delta gXg^{-1},\delta gu)
    \end{equation*}
    for $(g,\delta)\in\breve{G}(V)$, $X\in\g(V)$ and $u\in V$. This is a rational representation of $\breve{G}(V)$ (and $G(V)$ by taking restriction). We call this representation of $\breve{G}(V)$ (resp. $G(V)$) the \emph{enhanced standard representation} of $\breve{G}(V)$ (resp. $G(V)$).
    
   Note that if $G=\Sp_{2n}(\C)$ or $\RO_n(\C)$, then we have $G=G(E)$ and $\g=\g(E)$, where $E$ is equipped with the bilinear form $\langle\cdot,\cdot\rangle_E$ defined by
    \begin{equation*}
        \langle u,v\rangle_E=\begin{cases*}
            u^t\beta_{2n}v,&if $G=\Sp_{2n}(\C)$,\\
            u^t\alpha_nv,&if $G=\RO_n(\C)$.
        \end{cases*}
    \end{equation*}
    In this case, two definitions of the enhanced standard representation of $\breve{G}$ (resp. $G$) coincide.
	\subsection{Classical invariant theory}
	Let $H$ be a reductive group, acting on an affine variety $X$.
	It is known that the categorical quotient of $X$ by $H$ always exists. More precisely, $H\bac X=\Spec\left(\C[X]^H\right)$ and $\pi:X\to H\bac X$ is induced by the inclusion $\C[X]^H\hookrightarrow\C[X]$  (see \cite{PV}). Note that the morphism $\pi$ is surjective, and sends each $H$-invariant closed subset of $X$ onto a closed subset of $H\bac X$. Additionally,  every fiber of $\pi$ contains a unique closed orbit.
	
	\begin{theorem}[Luna's criterion, \cf \cite{PV}*{Remark of Theorem 6.17}]
		\label{thm:Lunacriterion}
		Let $K$ be a reductive subgroup of $H$, and let $x\in X^K$. Then the orbit $Hx$ is closed if and only if the orbit $Z_H(K)x$ is closed.
	\end{theorem}
    
	\subsection{The algebra of invariants}
    
	The first step to classify closed orbits is to determine the algebra $\C[\g\times E]^{G}$ of invariants. 

    \begin{theorem}
		\label{thm:structure}
		The algebra $\C[\g\times E]^{G}$ is a polynomial ring with   \begin{equation*}
        \mathfrak A_G=\begin{cases*}
        \{\tr_1,\dots,\tr_n,\mu_0,\dots,\mu_{n-1}\},&if $G=\GL_{n}(\C)$,\\
            \{\tr_2,\tr_4,\dots,\tr_{2n},\eta_1,\eta_3,\dots,\eta_{2n-1}\},&if $G=\Sp_{2n}(\C)$,\\
            \{\tr_2,\tr_4,\dots,\tr_{2n},\eta_0,\eta_2,\dots,\eta_{2n}\},&if $G=\RO_{2n+1}(\C)$,\\
            \{\tr_2,\tr_4,\dots,\tr_{2n},\eta_0,\eta_2,\dots,\eta_{2n-2}\},&if $G=\RO_{2n}(\C)$.
        \end{cases*}
    \end{equation*} as a set of algebraic independent generators.
	\end{theorem}
	\begin{proof}
		When $G$ is a general linear group, this is \cite{NO}*{Theorem 2.1 (2)}. In the case of orthogonal or symplectic groups, the structure of $\C[\g\times E]^G$ is given in \cite{Sch}*{Table 3a.2, 3a.5 and Table 4a.3}.
	\end{proof}
Here, for every $i\ge 1$, 	$\tr_i$ denotes the polynomial on $\g\times E$ given by 
    \begin{equation*}
        \tr_i(X,u)=\tr(X^i).
    \end{equation*}
 For every $j\ge 0$, $\mu_j$ denotes the polynomial of $\gl_n(\C)\times \C^{n\times 1}\times \C^{1\times n}$ given by 
    \begin{equation*}
        \mu_j(X,u,v)=vX^ju.
    \end{equation*}
And, when $G=\Sp_{2n}(\C)$, $\RO_{2n+1}(\C)$ or $\RO_{2n}(\C)$, $\eta_j$ denotes the polynomial on $\g\times E$ defined by 
    \begin{equation*}
        \eta_j(X,u)=\langle X^ju,u\rangle_E.
    \end{equation*}

	\section{Closed orbits of the maximal dimension}
	
	In this section, we give a family of closed $G$-orbits in $\CN_\g\times E$, which has the maximal dimension.

   
    \subsection{The general linear case}
    Assume that $G=\GL_n(\C)$. By Theorem \ref{thm:structure}, we regard the quotient $$\GL_n(\C)\bac(\gl_n(\C)\times\C^{n\times 1}\times\C^{1\times n})$$ as a closed subset of $\C^{1\times 2n}$, so that the quotient morphism is given by
	\begin{align*}
		\pi:\gl_n(\C)\times\C^{n\times 1}\times\C^{1\times n}&\to\GL_n(\C)\bac(\gl_n(\C)\times\C^{n\times 1}\times\C^{1\times n})\subseteq\C^{1\times 2n}\\
		x&\mapsto(\tr_1(x),\dots,\tr_n(x),\mu_0(x),\dots,\mu_{n-1}(x)).
	\end{align*}

    Recall the subset  $\fN_{\GL_n(\C)}$  of $\CN_{\gl_n(\C)}\times\C^{n\times 1}\times\C^{1\times n}$ defined in the Introduction.
    \begin{prop}
		\label{prop:maximalclosedgenerallinear}
		Let $x=x(n,y_1,\dots,y_n)\in\fN_{\GL_n(\C)}$. Then the stabilizer of $x$ is trivial, and $\GL_n(\C)x=\pi^{-1}(\underbrace{0,\dots,0}_n,y_1,\dots,y_n)$ is a closed orbit.
	\end{prop}
	\begin{proof}
		It is clear that $\pi(x)=(0,\dots,0,y_1,\dots,y_n)$ and that the stabilizer of $x$ is trivial. It remains to show that the fiber
		\begin{align*}
			O&=\pi^{-1}(0,\dots,0,y_1,\dots,y_n)\\
			&=\{(X,u,v)\in\gl_n(\C)\times\C^{n\times 1}\times\C^{1\times n}\mid X\text{ is nilpotent and }vu=y_1,\dots,vX^{n-1}u=y_n\}
		\end{align*}
		is an orbit.
		
		Let $(X,u,v)\in O$. Since $X$ is nilpotent and $X^{n-1}\neq 0$, we may assume that $X=J_n$. Write
		\begin{equation*}
			u=(u_1,\dots,u_n)^t\quad\text{and}\quad v=(v_1,\dots,v_n).
		\end{equation*}
		Then the condition that $vu=y_1,\dots,vJ_n^{n-1}u=y_n$ is equivalent to the equation
		\begin{equation*}
			\phi(v)u=(y_1,\dots,y_n)^t
		\end{equation*}
		In particular, $y_n=v_1u_n\neq 0$ forces that $v_1\neq 0$ and $u_n\neq 0$. Then $g=\phi(v)\in\GL_n(\C)$. It is easy to verify that
		\begin{equation*}
			gJ_ng^{-1}=J_n,\quad gu=(y_1,\dots,y_n)^t\quad\text{and}\quad(v_1,\dots,v_n)g^{-1}=(1,0,\dots,0).
		\end{equation*}
		Therefore the fiber $O$ is an orbit, as required.
	\end{proof}
	\subsection{The symplectic and orthogonal cases}
    \label{section:closedmaximal}
	 Assume that $G=\Sp_{2n}(\C),\RO_{2n+1}(\C)$ or $\RO_{2n}(\C)$. Set $m=\dim_{\C}E$. Recall the map $r:\fN_G\to\Z_{\geq 0}$ defined in Subsection \ref{sec:closed}. Let $\fM_G$ be the set consisting of $x\in\fN_G$ such that $r(x)=n$.

	\begin{lm}
		\label{lm:stabilizermaximal}
		Let $x=(X,u)\in\fM_G$. If $G=\Sp_{2n}(\C)$ or $\RO_{2n+1}(\C)$, then $G_x$ is trivial. If $G=\RO_{2n}(\C)$, then $G_x=\{\pm 1\}$.
	\end{lm}
	\begin{proof}
		Let $V=\Span_{\C}\{X^iu\mid i\in\Z_{\geq 0}\}\subseteq E$. Then $\langle\cdot,\cdot\rangle_E|_V$ is non-degenerate and $G_x=G(V^{\bot})$. If $G=\Sp_{2n}(\C)$ or $\RO_{2n+1}(\C)$, then $V=E$ and hence $G_x$ is trivial. When $G=\RO_{2n}(\C)$, $V^{\bot}=\Span_{\C}\{\be_n-\be_{2n}\}$ is one-dimensional, and hence $G_x=\RO(V^{\bot})=\{\pm 1\}$.
	\end{proof}
	
	Set $d=m-n$. By Theorem \ref{thm:structure}, the quotient morphism $$\pi:\g\times E\to G\bac(\g\times E)=\C^{1\times m}$$ is given by
	\begin{equation*}
		\pi(x)=\begin{cases*}
			(\tr_2(x),\dots,\tr_{2n}(x),\eta_1(x),\dots,\eta_{2n-1}(x)),&if $m=2n$ and $G=\Sp_{2n}(\C)$,\\
			(\tr_2(x),\dots,\tr_{2n}(x),\eta_0(x),\dots,\eta_{2n}(x)),&if $m=2n+1$ and $G=\RO_{2n+1}(\C)$,\\
			(\tr_2(x),\dots,\tr_{2n}(x),\eta_0(x),\dots,\eta_{2n-2}(x)),&if $m=2n$ and $G=\RO_{2n}(\C)$.
		\end{cases*}
	\end{equation*}
	
	\begin{prop}
		\label{prop:closedmaximal}
		Let $y=(0,\dots,0,y_1,\dots,y_d)\in G\bac(\g\times E)$ with $y_d\neq 0$. Then $\pi^{-1}(y)$ is a closed $G$-orbit with a representative $x\in\fM_G$
	\end{prop}
	\begin{proof}
		We prove the proposition for $G=\Sp_{2n}(\C)$. The other cases are proved similarly. Let $(Y,v)\in\pi^{-1}(y)$ and write $v=(v_1,\dots,v_{2n})$. Then $Y$ is nilpotent. Since
		\begin{equation*}
			\langle Y^{2n-1}v,v\rangle_E=y_n\neq 0,
		\end{equation*}
		we have $Y^{2n-1}\neq 0$. By \cite{CM}*{Proposition 5.2.3 and recipe 5.2.2}, $Y$ lies in the nilpotent orbit corresponding to the partition $[2n]$, and then it is $\Sp_{2n}(\C)$-conjugate with some element $X$ which is as in \eqref{eq:nilpotentclosedspm} with $k=n$. Let $g\in\Sp_{2n}(\C)$ be such that $gYg^{-1}=X$ and let $u=gv$. Then
		\begin{equation*}
			(-1)^{n-1}u_{n+1}^2=\langle X^{2n-1}u,u\rangle_E=y_n\neq 0,
		\end{equation*}
		so $u_{n+1}\neq 0$. Now $$\dim\Sp_{2n}(\C).(Y,v)=\dim\Sp_{2n}(\C)x=\dim\Sp_{2n}(\C)=n(2n+1)$$ for every $(Y,v)\in\pi^{-1}(y)$, i.e., each orbit in $\pi^{-1}(y)$ has dimension $n(2n+1)$. So $\pi^{-1}(y)$ contains exactly one orbit, which implies the proposition.
	\end{proof}

	\begin{cor}
		\label{cor:closedmaximal}
		If $x\in\fM_G$, then $Gx$ is a closed $G$-orbit.
	\end{cor}
    
	\section{Proof of Theorem \ref{thm:closedorbitsnullcone}}\label{sec:4}
    In this section, we give a proof of Theorem \ref{thm:closedorbitsnullcone}.
    
	\subsection{Proof of Theorem \ref{thm:closedorbitsnullcone}: the general linear case}
	In this subsection, we assume that $G=\GL_n(\C)$. Recall the quotient morphism $$\pi:\gl_n(\C)\times\C^{n\times 1}\times\C^{1\times n}\to\GL_n(\C)\bac(\gl_n(\C)\times\C^{n\times 1}\times\C^{1\times n})=\C^{1\times 2n}.$$
	\begin{prop}
		\label{prop:closedorbitgln}
		Let $x=x(k,y_1,\dots,y_k)\in\fN_{\GL_n(\C)}$. Then $\GL_n(\C)x$ is the unique closed $\GL_n(\C)$-orbit in the fiber$$\pi^{-1}(\underbrace{0,\dots,0}_n,y_1,\dots,y_k,\underbrace{0,\dots,0}_{n-k}).$$
Furthermore, the stabilizer $\GL_n(\C)_x\simeq\GL_{n-k}(\C)$.
	\end{prop}
	\begin{proof}
		Let
		\begin{equation*}
			y=(\underbrace{0,\dots,0}_n,y_1,\dots,y_k,\underbrace{0,\dots,0}_{n-k}),\quad\text{with }y_k\neq 0.
		\end{equation*}
		Then $x\in\pi^{-1}(y)$. If $k=n$, the proposition follows from Proposition \ref{prop:maximalclosedgenerallinear}. If $k=0$, it is clear that $\{(0,0,0)\}$ is a closed orbit in $\pi^{-1}(0)$.
        
        Now assume that $1\leq k\leq n-1$. Let $H=\GL_n(\C)_x$. Then
		\begin{equation*}
			H=\left\{\diag(I_k,g)\mid g\in \GL_{n-k}(\C)\right\}\simeq\GL_{n-k}(\C).
		\end{equation*}
		We have $x\in(\gl_n(\C)\times\C^{n\times 1}\times\C^{1\times n})^H$ and
		\begin{equation*}
			Z_{\GL_n(\C)}(H)=\left\{\diag(g,cI_{n-k})\mid g\in\GL_k(\C)\text{ and }c\in\C^{\times}\right\}.
		\end{equation*}
		By Proposition \ref{prop:maximalclosedgenerallinear}, $Z_{\GL_n(\C)}(H)x$ is closed in the closed subset
		\begin{equation*}
			\left\{\left(\left(\begin{matrix}
				Y&\\&0_{(n-k)\times(n-k)}
			\end{matrix}\right),\left(\begin{matrix}
				u^{\prime}\\0_{(n-k)\times 1}
			\end{matrix}\right),(v^{\prime},0_{1\times(n-k)})\right)\mid Y\in\gl_k(\C),u^{\prime}\in\C^{k\times 1}\text{ and }v^{\prime}\in\C^{1\times k}\right\},
		\end{equation*}
		and hence $Z_{\GL_n(\C)}(H)x$ is closed. By Theorem \ref{thm:Lunacriterion}, $\GL_n(\C)x$ is a closed orbit.
	\end{proof}
	
    
    \begin{prop}
        \label{prop:closedorbitglndifferent}
        Let $x,x'\in\fN_{\GL_n(\C)}$. Then $\GL_n(\C)x=\GL_n(\C)x'$ if and only if $k=k'$ and $y_1=z_1,\dots,y_k=z_k$, where $x=x(k,y_1,\dots,y_k)$ and $x'=x(k',z_1,\dots,z_{k'})$.
    \end{prop}
    \begin{proof}
        This follows immediately from Theorem \ref{thm:structure}.
    \end{proof}
    
	\subsection{Proof of Theorem \ref{thm:closedorbitsnullcone}: the symplectic and orthogonal cases}
	In this subsection, we assume that $G=\Sp_{2n}(\C),\RO_{2n+1}(\C)$ or $\RO_{2n}(\C)$. Recall the quotient morphism
	\begin{equation*}
		\pi:\g\times E\to G\bac (\g\times E)=\C^{1\times m},
	\end{equation*}
	where $m=\dim_{\C}E$.
	
	In the case that $G=\RO_{2n}(\C)$, we need an easy lemma.
	
	\begin{lm}
		\label{lm:evenorthogonalrestriction}
		Assume that $G=\RO_{2n}(\C)$. Let $x=(X,u)\in\fM_G$ and let $V=\Span_{\C}\{X^iu\mid i\in\Z_{\geq 0}\}$. We have $X\in\g(V)$.
	\end{lm}
	\begin{proof}
		Since $X|_V\in\End_{\C}(V)$, we have $X|_{V^{\bot}}\in\End_{\C}(V^{\bot})$. It suffices to show that $X|_{V^{\bot}}=0$, which follows from $\langle Xv,v\rangle_E=0$ for $v\in V^{\bot}$ since $\dim_{\C}V^{\bot}=1$.
	\end{proof}
	
	\begin{prop}
		\label{prop:closedorbit}
		Let $x=(X,u)\in\fN_{G}$ with $r(x)=k$. Then $Gx$ is a closed orbit, and the stabilizer
		\begin{equation}
			\label{eq:stabilizer}
			G_x\simeq\begin{cases*}
				\Sp_{2(n-k)}(\C),&if $G=\Sp_{2n}(\C)$,\\
				\RO_{2(n-k)}(\C),&if $G=\RO_{2n+1}(\C)$,\\
				\RO_{2(n-k)+1}(\C),& if $G=\RO_{2n}(\C)$,
			\end{cases*}
		\end{equation}
		if $k\neq 0$.
	\end{prop}
	\begin{proof}
		If $k=0$, then $x=(0,0)$ and $Gx$ is obviously a closed orbit. If $k=n$, the proposition follows from Corollary \ref{cor:closedmaximal}.
		
		Assume that $1\leq k\leq n-1$. Let
		\begin{equation}
			\label{eq:subspace}
			W_k=\begin{cases*}
				\Span_{\C}\{\be_1,\dots,\be_k,\be_{n+1},\dots,\be_{n+k}\},&if $G=\Sp_{2n}(\C)$ or $\RO_{2n}(\C)$,\\
				\Span_{\C}\{\be_1,\be_2,\dots,\be_{k+1},\be_{n+1},\dots,\be_{n+k}\},&if $G=\RO_{2n+1}(\C)$,
			\end{cases*}
		\end{equation}
		and let $V_k=\Span_{\C}\{u,Xu,\dots,X^{m-1}u\}$. Then $\langle\cdot,\cdot\rangle_E$ is non-degenerated on $V_k$, and $E=V_k\oplus V_k^{\bot}$. Thus $G_x=G(V_k^{\bot})$. Note that $V_k=W_k$ if $G=\Sp_{2n}(\C)$ or $\RO_{2n+1}(\C)$, and $V_k$ is a subspace of $W_k$ of codimension 1 if $G=\RO_{2n}(\C)$. This implies \eqref{eq:stabilizer}.
		
		Write $H=G_x$. Then $x\in(\g\times E)^H$ and $Z_{G}(H)=G(V_k)$. By Lemma \ref{lm:evenorthogonalrestriction}, $x\in\g(V_k)\times V_k$. We have seen that $Z_{G}(H)x$ is a closed orbit in the closed subset $\g(V_k)\times V_k$ of $\g\times E$, so $Gx$ is a closed orbit by Theorem \ref{thm:Lunacriterion}.
	\end{proof}
	\begin{prop}
		\label{prop:closedorbit2}
		Every closed $G$-orbit in $\CN_{\g}\times E$ has a representative $x\in\fN_{G}$.
	\end{prop}
	\begin{proof}
		Let
		\begin{equation*}
			y=(\underbrace{0,\dots,0}_n,\underbrace{y_1,\dots,y_d}_d,0,\dots,0)\in\C^{1\times m}
		\end{equation*}
		with $d\geq 1$ and $y_d\neq 0$. Let
		\begin{equation*}
			k=\begin{cases*}
				d,&if $G=\Sp_{2n}(\C)$ or $\RO_{2n}(\C)$,\\
				d-1,&if $G=\RO_{2n+1}(\C)$.
			\end{cases*}
		\end{equation*}
		We need to show that $\pi^{-1}(y)$ contains an element $x=(X,u)\in\fN_G$ with $r(x)=k$. If $k=0$, then $y=0$ and $\{(0,0)\}$ is the closed orbit in $\pi^{-1}(0)$. If $k=n$, this follows from Proposition \ref{prop:closedmaximal}.
		
		Assume that $1\leq k\leq n-1$. Let $W_k$ be defined by \eqref{eq:subspace}. Then we have a natural embedding $\g(W_k)\hookrightarrow\g$. Denote by
		\begin{equation*}
			\pi_k:\g(W_k)\times W_k\to G(W_k)\bac(\g(W_k)\times W_k)=\C^{1\times (k+d)}
		\end{equation*}
		the quotient morphism. By Proposition \ref{prop:closedmaximal}, $\pi_k^{-1}(0,\dots,0,y_1,\dots,y_d)$ is a closed $G(W_k)$-orbit containing a representative $x'=(X^{\prime},u^{\prime})\in\fN_{G(W_k)}$ with such that $r(x')=k$. Write
		\begin{equation*}
			u^{\prime}=(u_1,\dots,u_d,u_{n+1},\dots,u_{n+k})^t.
		\end{equation*}
		Let $x=(X,u)\in\fN_G$ such that $r(x)=k$ and
		\begin{equation*}
			u=(u_1,\dots,u_d,\underbrace{0,\dots,0}_{n-k},u_{n+1},\dots,u_{n+k},\underbrace{0,\dots,0}_{n-k})^t.
		\end{equation*}
		Then $x=(X,u)\in\pi^{-1}(y)$.
	\end{proof}
    If $x\in\fN_{G}$ with $r(x)=k$, set
    \begin{equation}
    	\label{eq:invariant}
        \eta(x)=\begin{cases*}
        	(\eta_1(x),\eta_3(x),\dots,\eta_{2k-1}(x)),&if $G=\Sp_{2n}(\C)$,\\
        	(\eta_0(x),\eta_2(x),\dots,\eta_{2k}(x)),&if $G=\RO_{2n+1}(\C)$,\\
        	(\eta_0(x),\eta_2(x),\dots,\eta_{2k-2}(x)),&if $G=\RO_{2n}(\C)$.
        \end{cases*}
    \end{equation}
    \begin{prop}
        \label{prop:closeddifferent}
        Let $x,x'\in\fN_{G}$. Then $Gx=Gx'$ if and only if $r(x)=r(x')$ and $\eta(x)=\eta(x')$.
    \end{prop}
    \begin{proof}
        This follows from Theorem \ref{thm:structure}.
    \end{proof}
    
    We now finish the proof of Theorem \ref{thm:closedorbitsnullcone}.
    \begin{proof}[Proof of Theorem \ref{thm:closedorbitsnullcone}]
        The theorem follows from Propositions \ref{prop:closedorbitgln}, \ref{prop:closedorbit} and \ref{prop:closedorbit2}.
    \end{proof}
    
	\section{Proof of Theorems \ref{thm:closedorbit} and \ref{thm:closedorbitsmvw}}\label{sec:5}
This section is devoted to a proof of Theorems  \ref{thm:closedorbit} and \ref{thm:closedorbitsmvw}.	For every $X\in \g$, write  
	\begin{equation*}
		X=X_{\rs}+X_{\rn}
	\end{equation*}
    for the Jordan decomposition of $X$, where
	$X_{\rs}$ is semisimple and $X_{\rn}$ is nilpotent. 
Write $L_{X_{\rs}}=Z_{G}(X_{\rs})$ for the centralizer of $X_{\rs}$ in $G$, and write
	\begin{equation*}
		\CN_{X_{\rs}}=\{Y\in\CN_{\g}\mid [X_{\rs},Y]=0\}.
	\end{equation*}
    Note that $\CN_{X_{\rs}}$ is $L_{X_{\rs}}$-stable.
	\begin{lm}
		\label{lm:closedorbitnilpotent}
		Let $O$ be a closed orbit in $\g\times E$ and let $(X,u)\in O$. Set $O_{\rn}=L_{X_{\rs}}.(X_{\rn},u)$. Then $O_{\rn}$ is a closed $L_{X_{\rs}}$-orbit in $\CN_{X_{\rs}}\times E$, and
		\begin{equation}
			\label{eq:closedorbitnilpotent}
			O\cap \left((X_{\rs}+\CN_{X_{\rs}})\times E\right)=(X_{\rs},0)+O_{\rn}.
		\end{equation}
	\end{lm}
	\begin{proof}
	Note that	the closeness of $O_{\rn}$ follows from the equality \eqref{eq:closedorbitnilpotent}. Thus, we only need to prove  this equality. 
		
		It is obvious that $(X_{\rs},0)+O_{\rn}\subseteq O\cap \left((X_{\rs}+\CN_{X_{\rs}})\times E\right)$. On the other hand, let $Y\in\CN_{\g}$ be such that $(X_{\rs}+Y,v)\in O$. Then   $(X_{\rs}+Y,v)=g.(X_{\rs}+X_{\rn},u)$ for some $g\in G$. Note that
		\begin{equation*}
			gX_{\rs}g^{-1}+gX_{\rn}g^{-1}
		\end{equation*}
		is the Jordan decomposition of $X_{\rs}+Y$, which forces that $g\in L_{X_{\rs}}$ and $g.(X_{\rn},u)=(Y,v)$. Thus, $(X_{\rs}+Y,v)\in(X_{\rs},0)+O_{\rn}$ and so $O\cap \left((X_{\rs}+\CN_{X_{\rs}})\times E\right)\subseteq(X_{\rs},0)+O_{\rn}$, as required. 
	\end{proof}
	
	\subsection{Proof of Theorem \ref{thm:closedorbit}: the general linear case}
    In this subsection, we assume that $G=\GL_n(\C)$. First, by Lemma \ref{lm:closedorbitnilpotent}, we have the following result.
	\begin{prop}
		\label{prop:closedorbitnilpotentgln}
		Each closed $\GL_n(\C)$-orbit of $\gl_n(\C)\times\C^{n\times 1}\times\C^{1\times n}$ has a representative in $\fX_{\GL_n(\C)}$.
	\end{prop}
    
    In what follows, we are going to show that   $\GL_n(\C)x$ is  closed  for every $x\in\fX_{\GL_n(\C)}$.
    
	\begin{lm}
		\label{lm:regularsemisimplegln}
		Let $x=(X,u,v)\in\gl_n(\C)\times\C^{n\times 1}\times\C^{1\times n}$. If $\{u,Xu,\dots,X^{n-1}u\}$ is a basis of $\C^{n\times 1}$ and $\{v,vX,\dots,vX^{n-1}\}$ is a basis of $\C^{1\times n}$, then $\GL_n(\C)x$ is a closed orbit and the centralizer of $x$ is trivial.
	\end{lm}
	\begin{proof}
		This is proved in \cite{RS}*{Theorem 6.3}.
	\end{proof}
    
	\begin{prop}
		\label{prop:closedorbitmaximalgln}
		Let $x=(X,u,v)$ such that
		\begin{equation*}
			X=\diag(c_1I_{n_1}+J_{n_1},\dots,c_bI_{n_b}+J_{n_b})\in\gl_n(\C)
		\end{equation*}
		with $c_i\neq c_j$ for $1\leq i\neq j\leq b$,
		\begin{equation*}
			u=\left(\left(u^{(1)}\right)^t,\dots,\left(u^{(b)}\right)^t\right)^t\in\C^{n\times 1},\quad\text{and}\quad v=(v^{(1)},\dots,v^{(b)})\in\C^{1\times n}
		\end{equation*}
		with $u^{(i)}=(u_1^{(i)},\dots,u_{n_i}^{(i)})^t\in\C^{n_i\times 1}$ ($u_{n_i}^{(i)}\neq 0$) and $v^{(i)}=(1,0,\dots,0)\in\C^{1\times n_i}$ for $1\leq i\leq b$. Then $\GL_n(\C)x$ is a closed orbit.
	\end{prop}
	\begin{proof}
		Let $V=\Span_{\C}\{u,Xu,\dots,X^{n-1}u\}$ and $W=\Span_{\C}\{v,vX,\dots,vX^{n-1}\}$. We claim that $V=\C^{n\times 1}$ and $W=\C^{1\times n}$. The proposition follows from the claim by Lemma \ref{lm:regularsemisimplegln}.
		
		Now we prove the claim. Let
		\begin{align*}
			S_i&=\left\{(X-c_iI_n)^k\prod_{l=i+1}^b(X-c_lI_n)^{n_l}u\mid k=0,\dots n_i-1\right\},\\
			T_i&=\left\{\prod_{l=1}^{i-1}(X-c_lI_n)^{n_l}(X-c_iI_n)^ku\mid k=0,\dots,n_i-1\right\}
		\end{align*}
		for $1\leq i\leq b$. Put $S=\bigcup_{i=1}^bS_i$ and $T=\bigcup_{i=1}^bT_i$. Then $S\subseteq V$ and $T\subseteq W$. It is easy to verify that $S$ is a basis of $\C^{n\times 1}$ and $W$ is a basis of $\C^{1\times n}$. Therefore $V=\C^{n\times 1}$ and $W=\C^{1\times n}$.
	\end{proof}
	\begin{prop}
		\label{prop:closedgln}
		For every $x\in\fX_{\GL_n(\C)}$, the orbit $\GL_n(\C)x$ is closed.
	\end{prop}
	\begin{proof}
        Write $x=(X,u,v)$ as in \eqref{eq:closedorbitnilpotentgln}. Then for $1\leq i\leq b$, we have
        \begin{equation*}
			N_i=\diag(
				J_{k_i},0_{(n_i-k_i)\times(n_i-k_i)})\quad\text{for some }0\leq k_i\leq n_i,
        \end{equation*}
        $u^{(i)}=(u_1^{(i)},\dots,u_{k_i}^{(i)},0,\dots,0)^t\in\C^{n_i\times 1}$ ($u_{k_i}^{(i)}\neq 0$), and $v^{(i)}=(1,0,\dots,0)\in\C^{1\times n_i}$
    
		View $\C^{n\times 1}$ as the direct sum of $\C^{n_1\times 1},\dots,\C^{n_b\times 1}$. Denote by $\{\be_1^{(i)},\dots,\be_{n_i}^{(i)}\}$ the standard basis of $\C^{n_i\times 1}$. Then $\{\be_1^{(1)},\dots,\be_{n_1}^{(1)},\dots,\be_1^{(b)},\dots,\be_{n_b}^{(b)}\}$ is a standard basis of $\C^{n\times 1}$. Put
		\begin{equation*}
			H=\GL_n(\C)_x=\left\{\diag(I_{k_1},g_1,\dots,I_{k_b},g_b)\mid g_1\in\GL_{n_1-k_1}(\C),\dots,g_b\in\GL_{n_b-k_b}(\C)\right\}.
		\end{equation*}
		Then $x\in (\gl_n(\C)\times\C^{n\times 1}\times\C^{1\times n})^H$. Let $V=\Span_{\C}\{\be_1^{(1)},\dots,\be_{k_1}^{(1)},\dots,\be_1^{(b)},\dots,\be_{k_b}^{(b)}\}$, and $V_i=\Span_{\C}\{\be_{k_i+1}^{(i)},\dots,\be_{n_i}^{(i)}\}$ for $1\leq i\leq b$. Then
		\begin{equation*}
			Z_{\GL_n(\C)}(H)=\GL(V)\times\C^{\times}\cdot I_{V_1}\times\dots\times\C^{\times}\cdot I_{V_b}.
		\end{equation*}
		By Proposition \ref{prop:closedorbitmaximalgln}, $Z_{\GL_n(\C)}(H)x$ is a closed orbit in $\gl_n(\C)\times\C^{n\times 1}\times\C^{1\times n}$. Therefore $\GL_n(\C)x$ is a closed orbit by Theorem \ref{thm:Lunacriterion}.
	\end{proof}
	\subsection{Proof of Theorem \ref{thm:closedorbit}: the symplectic and orthogonal cases}
	Assume that $G=\Sp_{2n}(\C)$, $\RO_{2n+1}(\C)$ or $\RO_{2n}(\C)$. Set $m=\dim_{\C}E$. The following result is obvious.
	
	\begin{lm}
		\label{lm:0eigenspace}
		Let $X\in\g$. If the generalized 0-eigenspace $E_0$ of $X$ is nonzero, then the restriction of the bilinear form $\langle\cdot,\cdot\rangle_E$ on $E_0$ is non-degenerate. 
	\end{lm}
		
By Lemma \ref{lm:closedorbitnilpotent}, we obtain the following result.

	\begin{prop}
		\label{prop:closedorbitnilpotent}
		Each closed $G$-orbit in $\g\times E$ has a representative $x\in\fX_G$.
	\end{prop}
    In what follows, we are going to show that   $Gx$ is  closed  for every $x\in\fX_{G}$.
    
    \begin{lm}
    	\label{lm:closedmaximal}
    	Let $x=(X,u)\in\g\times E$. If $\{u,Xu,\dots,X^{m-1}u\}$ is a basis of $E$, then $Gx$ is a closed orbit and the stabilizer of $x$ is trivial.
    \end{lm}
    \begin{proof}
    	This is proved in \cite{RS}*{Theorem 17.1}.
    \end{proof}
    
	\begin{lm}
		\label{lm:closedorbitmaximalso}
		Let $x=(X,u)\in\fX_G$. Write $X,u$ as in \eqref{eq:closedsom1} and \eqref{eq:closedsov}. Assume that $x$ satisfies the following conditions:
        \begin{itemize}
            \item $N_i=J_{n_i}$, $v^{(i)}=(1,0,\dots,0)$, $u^{(i)}=(u_1^{(i)},\dots,u_{n_i}^{(i)})^t$ with $u_{n_i}^{(i)}\neq 0$ for $1\leq i\leq b$,
            \item $\left(\left(\begin{matrix}
				N_0^{(1)}&N_0^{(2)}\\&N_0^{(3)}
			\end{matrix}\right),\left(\begin{matrix}
			    u^{(0)}\\v^{(0)}
			\end{matrix}\right)\right)\in\fM_{G_0}$, where $G_0$ is defined by \eqref{eq:subgroup}.
        \end{itemize}
		Let $V=\Span_{\C}\{u,Xu,\dots,X^{m-1}u\}$. Then $V$ is an orthogonal subspace of codimension 1, if $G=\RO_{2n}(\C)$ and $n_0\neq 0$. Otherwise $V=E$.
	\end{lm}
	\begin{proof}
		Let $E_0$ be the generalized 0-eigenspace of $X$, $E_i$ the generalized $c_i$-eigenspace, and $E_i^*$ the generalized $-c_i$-eigenspace, for $1\leq i\leq b$. Put
		\begin{equation*}
			E^{\prime}=(E_1\oplus E_1^*)\oplus\dots\oplus(E_b\oplus E_b^*).
		\end{equation*}
		Let $p$ be the projection from $E$ to $E'$ with respect to the decomposition $E=E_0\oplus E'$.
		
		Let $v=\prod_{i=1}^{b}(X+c_iI_n)^{n_i}u$,
		\begin{equation*}
			S_i^*=\left\{(X+c_iI_n)^k\prod_{l=i+1}^{b}(X+c_lI_n)^{n_l}u\mid k=0,\dots,n_i-1\right\},
		\end{equation*}
		and
		\begin{equation*}
			S_i=\left\{(X-c_iI_n)^k\prod_{l=i+1}^{b}(X-c_lI_n)^{n_l}v\mid k=0,\dots,n_i-1\right\}
		\end{equation*}
		for $1\leq i\leq b$. Put $S=\bigcup_{i=1}^b(S_i^*\cup S_i)$. By the proof of Proposition \ref{prop:closedgln}, $p(S)$ is a basis of $E^{\prime}$. Let
		\begin{equation*}
			u_0=\prod_{i=1}^b(X-c_iI_n)^{n_i}(X+c_iI_n)^{n_i}u
		\end{equation*}
		Then $u_0\in E_0$ and
		\begin{equation*}
			u_0=(s_1,\dots,s_{n_0^{\prime}},\underbrace{0,\dots,0}_{n_1+\dots+n_b},s_{n_0^{\prime}+1},\dots,s_{n_0^{\prime}+n_0},\underbrace{0,\dots,0}_{n_1+\dots+n_b})
		\end{equation*}
		with $s_{n_0^{\prime}+1}=(-1)^bc_1^2\dots c_b^2v_1^{(0)}\neq 0$.
		%
		
		Let $V=\Span_{\C}\{u,Xu,\dots,X^{m-1}u\}$ and $T=\{u_0,Xu_0,\dots,X^{n_0^{\prime}+n_0-1}u_0\}$. Assume that $G=\Sp_{2n}(\C)$ or $\RO_{2n+1}(\C)$. Then $T$ is a basis of $E_0$. Since $S$ and $T$ are contained in $V$, we have $V=E$.
		
		Assume that $G=\RO_{2n}(\C)$. If $n_0=0$, then $V=E$. Now assume that $n_0\geq 1$. Let $V_0=\Span_{\C}\{u_0,Xu_0,\dots,X^{2n_0-2}u_0\}$. Then $V=V_0\oplus E^{\prime}$, and hence $V^{\bot}$ equals the orthogonal complement of $V_0$ in $E_0$. Therefore $\dim_{\C}V^{\bot}=1$.
	\end{proof}
	\begin{prop}
		\label{prop:closedorbitso}
		For every $x\in\fX_G$, $Gx$ is a closed orbit.
	\end{prop}
	\begin{proof}
		Write $x=(X,u)$. Let $V=\Span_{\C}\{u,Xu,\dots,X^{m-1}u\}$. Then $x\in\g(V)\times V$. By Lemma \ref{lm:closedorbitmaximalso}, the restriction of the bilinear form $\langle\cdot,\cdot\rangle_E$ to $V$ is nondegenerate. Let $H=G_x$. Then $H=G(V^{\bot})$ and hence $Z_G(H)=G(V)$. By Lemma \ref{lm:closedmaximal}, $Z_G(H)x$ is a closed orbit in the closed subset $\g(V)\times V$. Therefore $Gx$ is a closed orbit by Theorem \ref{thm:Lunacriterion}.
	\end{proof}
    \subsection{Proof of Theorem \ref{thm:closedorbitsmvw}}
Here we give a proof of Theorem \ref{thm:closedorbitsmvw}, which states that every closed $G$-orbit in $\g\times E$ is $\breve{G}$-stable.

The following result is obvious.

	\begin{lm}
		\label{lm:mvwclosed}
		Let $H$ be a reductive group acting on an affine variety $X$, and let $K$ be a closed subgroup of $H$ which has index $2$. Then each closed $K$-orbit is $H$-stable if and only if $\C[X]^H=\C[X]^K$.
	\end{lm}
		

In view of Lemma \ref{lm:mvwclosed}, Theorem \ref{thm:closedorbitsmvw} is implied by the following result. 
	\begin{lm}
		\label{lm:mvwinvariant}
		We have $\C[\g\times E]^{\breve{G}}=\C[\g\times E]^{G}$.
	\end{lm}
	\begin{proof}
		It follows immediately from Theorem \ref{thm:structure}.

	\end{proof}

		
		
		

	\section{Proof of Theorem \ref{thm:mainmvwsimple}}\label{sec:6}
    In this section, we give a proof of Theorem \ref{thm:mainmvwsimple}.

    Let $O$ be a closed $G$-orbit, and let $x\in O$. Let $\g_x$ be the Lie algebra of the stabilizer $G_x$. Then we have a $G_x$-module isomorphism
	\begin{equation*}
		T_xO=\g x\simeq\g/\g_x,
	\end{equation*}
	where $G_x$ acts by the adjoint action on $\g/\g_x$. Therefore the normal space
	\begin{equation}
		\label{eq:descent1}
		N_{O,x}^{\g\times E}\simeq\g_x\times E
	\end{equation}
	as rerpesentations of $G_x$, where $G_x$ acts on $\g_x$ by the adjoint action and acts on $E$ via the inclusion $G_x\subseteq G$.
	
	\subsection{Descents at $x\in\fN_{G}$ for $G=\GL_n(\C)$}
    In this subsection, we assume that $G=\GL_n(\C)$. Fix $x=x(k,y_1,\dots,y_k)\in\fN_{G}$.
    
    The following result is obvious.
    
    \begin{lm}
    	\label{lm:descentgln}
    	Let $h_1,h_2\in\GL_m(\C)$ be symmetric. For $i=1,2$, consider representations $(\rho_i,\C^{m\times 1}\times\C^{1\times m})$ and $(\omega_i,\gl_m(\C))$ of $\{\pm 1\}$ defined by
    	\begin{equation*}
    		\rho_i(-1)(u,v)=(-h_iv^t,-u^th_i^{-1})\quad\text{and}\quad\omega_i(-1)X=h_iX^th_i^{-1}.
    	\end{equation*}
    	Then
    	\begin{equation*}
    		\rho_1\simeq\rho_2,\quad (u,v)\mapsto(h_2h_1^{-1}u,v)
    	\end{equation*}
    	and
    	\begin{equation*}
    		\omega_1\simeq\omega_2,\quad X\mapsto h_2h_1^{-1}X.
    	\end{equation*}
    	In particular, $\rho_i\simeq\triv^m\oplus\sgn^m$ and $\omega_i\simeq\triv^{m(m+1)/2}\oplus\sgn^{(m-1)m/2}$ for $i=1,2$, where $\sgn$ denotes the sign character.
    \end{lm}
    
    Now we prove Theorem \ref{thm:mainmvwsimple} for $G=\GL_n(\C)$ and $x\in\fN_{G}$.
    \begin{prop}
    	\label{prop:descentglnmvw}
    	The stabilizer $\breve{\GL}_n(\C)_x\simeq\breve{\GL}_{n-k}(\C)$, and
    	\begin{equation}
    		\label{eq:descentgln}
    		N_{O,x}^{\g\times E}\simeq(\gl_{n-k}(\C)\times\C^{(n-k)\times 1}\times\C^{1\times(n-k)})\oplus\triv^{2k}
    	\end{equation}
    	as representations of $\breve{\GL}_{n-k}(\C)$.
    \end{prop}
    \begin{proof}
    	We have seen that
    	\begin{equation*}
    		\GL_n(\C)_x=\left\{\diag(I_k,g)\mid g\in\GL_{n-k}(\C)\right\}\simeq\GL_{n-k}(\C).
    	\end{equation*}
    	By Theorem \ref{thm:closedorbitsmvw}, we have $\breve{\GL}_n(\C)x=\GL_{n}(\C)x$, and hence $\GL_n(\C)_x$ is a subgroup of $\breve{\GL}_n(\C)_x$ of index two.
    	
    	Let $y=(y_1,\dots,y_k)$ and $g_0=\diag(-\psi(y),I_{n-k})\in\GL_n(\C)$. Then we have $(g_0,-1)\in\breve{\GL}_n(\C)_x$, and the map $\breve{\GL}_n(\C)_x\to\breve{\GL}_{n-k}(\C)$ defined by
    	\begin{equation*}
    		(\diag(I_k,g),1)\mapsto g\quad\text{and}\quad(g_0,-1)\mapsto(I_{n-k},-1).
    	\end{equation*}
    	is an isomorphism.
    	
    	Now we consider the action of $\breve{\GL}_{n-k}(\C)$ on $N_{O,x}^{\g\times E}$. For $z\in\g\times E$, write
    	\begin{equation*}
    		z=\left(\left(\begin{matrix}
    			A&B\\C&D
    		\end{matrix}\right),\left(\begin{matrix}
    		    u\\u'
    		\end{matrix}\right),(v,v')\right),
    	\end{equation*}
    	where $A\in\gl_k(\C),\ B\in\Mat_{k\times(n-k)}(\C),\ C\in\Mat_{(n-k)\times k}(\C),\ D\in\gl_{n-k}(\C),\ u\in\C^{k\times 1},\ u'\in\C^{(n-k)\times 1},\ v\in\C^{1\times k}$ and $v'\in\C^{1\times (n-k)}$. The action of $\breve{\GL}_{n-k}(\C)$ on $\g\times E$ is given by
    	\begin{equation*}
    		(g,1).z=\left(\left(\begin{matrix}
    			A&Bg^{-1}\\gC&gDg^{-1}
    		\end{matrix}\right),\left(\begin{matrix}
    			u\\gu'
    		\end{matrix}\right),(v,v'g^{-1})\right)
    	\end{equation*}
    	and
    	\begin{equation*}
    		(I_{n-k},-1).z=\left(\left(\begin{matrix}
    			\psi(y)A^t\psi(y)^{-1}&\psi(y)C^t\\B^t\psi(y)^{-1}&D
    		\end{matrix}\right),\left(\begin{matrix}
    			-\psi(y)v^t\\-v'^t
    		\end{matrix}\right),(-u^t\psi(y)^{-1},-u'^t)\right).
    	\end{equation*}
    	The tangent space
    	\begin{equation*}
    		T_xO=\left\{w=\left(\begin{matrix}
    			A&B\\C&0_{(n-k)\times(n-k)}
    		\end{matrix}\right)\mid A\in\gl_k(\C),\ B\in\Mat_{k\times(n-k)}(\C),\text{ and }C\in\Mat_{(n-k)\times k}(\C)\right\}
    	\end{equation*}
    	with the action of $\breve{\GL}_{n-k}(\C)$ given by
    	\begin{equation*}
    		(g,1).w=\left(\begin{matrix}
    			A&Bg^{-1}\\gC&0_{(n-k)\times(n-k)}
    		\end{matrix}\right)\quad\text{and}\quad(I_{n-k},-1).w=\left(\begin{matrix}
    		-\psi(y)A^t\psi(y)^{-1}&\psi(y)C^t\\B^t\psi(y)^{-1}&0_{(n-k)\times(n-k)}
    		\end{matrix}\right).
    	\end{equation*}
    	Comparing these two representations, we obtain \eqref{eq:descentgln} by Lemma \ref{lm:descentgln}.
    \end{proof}
    
    \subsection{Descents at $x\in\fN_{G}$ for $G=\Sp_{2n}(\C)$ or $\RO_m(\C)$}
    In this subsection, assume that $G=\Sp_{2n}(\C),\RO_{2n+1}(\C)$ or $\RO_{2n}(\C)$. Let $m=\dim_{\C}E$. Fix $x=(X,u)\in\fN_{G}$ with $r(x)=k$. Let $V=\Span_{\C}\{u,Xu,\dots,X^{m-1}u\}$ and let
    \begin{equation*}
    	(\gamma,\gamma')=\begin{cases*}
    		(0,1),&if $G=\Sp_{2n}(\C)$, or $(X,u)=(0,0)$,\\
    		(1,0),&if $G=\RO_{2n+1}(\C)$ and $(X,u)\neq(0,0)$,\\
    		(-1,0),&if $G=\RO_{2n}(\C)$ and $(X,u)\neq(0,0)$.
    	\end{cases*}
    \end{equation*}
    
    \begin{lm}
    	\label{lm:descent2}
    	Assume that $V\neq\{0\}$. Let $h'\in\g(V)$ be defined by
    	\begin{equation*}
    		h'(X^iu)=(-1)^{i+1}X^iu\quad\text{for }0\leq i\leq\dim_{\C}V-1.
    	\end{equation*}
    	Consider the representation of $\{\pm 1\}$ on $\g(V)$ given by $(-1).Y=-h'Yh'^{-1}$. Then
    	\begin{equation*}
    		\g(V)\simeq\triv^{k(k+\gamma+\gamma')}\oplus\sgn^{(k+\min\{\gamma,\gamma'\})^2}.
    	\end{equation*}
    \end{lm}
    \begin{proof}
    	We prove for $G=\Sp_{2n}(\C)$. The other cases are proved similarly. Let $J_e=\{0\leq i\leq 2k-1\mid i\text{ is even}\}$ and $J_o=\{0\leq i\leq 2k-1\mid i\text{ is odd}\}$. Identify $Y\in\g(V)$ with a matrix with respect to the basis $\{X^iu\mid i\in J_e\}\cup\{X^iu\mid i\in J_o\}$. Then
    	\begin{equation*}
    		\g(V)=\left\{\left(\begin{matrix}
    			Y_{1}&Y_{2}\\Y_{3}&Y_{4}
    		\end{matrix}\right)\mid Y_{i}\in\gl_k(\C),\ AY_1+Y_4^tA=AY_2-Y_2^tA=AY_3-Y_3^tA=0\right\},
    	\end{equation*}
    	where
    	\begin{equation*}
    		A=\left(\langle X^iu,X^ju\rangle\right)_{i\in J_e,j\in J_o}=\psi(\eta(x))\quad(\eta\text{ is defined by \eqref{eq:invariant}}),
    	\end{equation*}
    	and
    	\begin{equation*}
    		(-1).\left(\begin{matrix}
    			Y_{1}&Y_{2}\\Y_{3}&Y_{4}
    		\end{matrix}\right)=\left(\begin{matrix}
    		-Y_{1}&Y_{2}\\Y_{3}&-Y_{4}
    		\end{matrix}\right),
    	\end{equation*}
    	which implies the lemma.
    \end{proof}
    
    \begin{prop}
    	\label{prop:descentmvw}
    	The stabilizer $\breve{G}_x\simeq\breve{G}(V^{\bot})$ and 
    	\begin{equation}
    		\label{eq:descentmvw}
    		N_{O,x}^{\g\times E}\simeq(\g(V^{\bot})\times V^{\bot})\oplus\triv^{2k+\gamma},
    	\end{equation}
    	as representations of $\breve{G}(V^{\bot})$.
    \end{prop}
    \begin{proof}
    	If $x=(0,0)$ the proposition is trivial. Assume that $x\neq (0,0)$. Note that $\breve{G}_x$ preserves $V$ and $V^{\bot}$. We have a homomorphism
    	\begin{equation*}
    		\breve{G}_x\to\breve{G}(V^{\bot}),\quad (g,\delta)\mapsto (g|_{V^{\bot}},\delta),
    	\end{equation*}
    	which has an inverse
    	\begin{equation*}
    		\breve{G}(V^{\bot})\to\breve{G}_x\subseteq\GL(E)\times\{\pm 1\},\quad(h,\delta)\mapsto(\tilde{h},\delta)
    	\end{equation*}
    	such that
    	\begin{equation*}
    		\tilde{h}|_{V^{\bot}}=h\quad\text{and}\quad\tilde{h}(X^iu)=\delta^{i+1}X^iu,\text{ for }0\leq i\leq\dim_{\C}V-1.
    	\end{equation*}
    	Therefore $\breve{G}_x\simeq\breve{G}(V^{\bot})$.
    	
    	Now we consider the action of $\breve{G}(V^{\bot})$ on $N_{O,x}^{\g\times E}$. As vector spaces,
    	\begin{equation*}
    		\g\times E\simeq(\g(V)\times V)\oplus(\g(V^{\bot})\times V^{\bot})\oplus\Hom_{\C}(V,V^{\bot}),
    	\end{equation*}
    	and
    	\begin{equation*}
    		T_xO\simeq\g(V)\oplus\Hom_{\C}(V,V^{\bot}).
    	\end{equation*}
    	The actions of $\breve{G}(V^{\bot})$ are given by
    	\begin{equation}
    		\label{eq:descentmvw1}
    		(h,\delta).((Y,v),(Y',v'),Z)=((\delta h'Yh'^{-1},\delta h'v),(\delta hY'h^{-1},\delta hv'),Z)\quad\text{on }\g\times E
    	\end{equation}
    	and
    	\begin{equation}
    		\label{eq:descentmvw2}
    		(h,\delta).(Y,Z)=(h'Yh'^{-1},Z)\quad\text{on }T_xO,
    	\end{equation}
    	where $h'=\tilde{h}|_V$. Comparing \eqref{eq:descentmvw1} and \eqref{eq:descentmvw2}, we obtain \eqref{eq:descentmvw} by Lemma \ref{lm:descent2}.
    \end{proof}
    
    \subsection{Descents at general $x\in\fX_{G}$}
    Fix $x\in\fX_{G}$. Write $x$ in the form \eqref{eq:closedorbitnilpotentgln} if $G=\GL_n(\C)$, and write $x=(X,u)$ in the form \eqref{eq:closedsom1} and \eqref{eq:closedsov} if $G=\Sp_{2n}(\C),\RO_{2n+1}(\C)$ or $\RO_{2n}(\C)$. Set
	\begin{equation*}
		x_i=(N_i,u^{(i)},v^{(i)})\quad\text{for }1\leq i\leq b,
	\end{equation*}
	If $G=\Sp_{2n}(\C),\RO_{2n+1}(\C)$ or $\RO_{2n}(\C)$, let
	\begin{equation*}
		x_0=\left(\left(\begin{matrix}
			N_1^{(0)}&N_2^{(0)}\\&N_3^{(0)}
		\end{matrix}\right),\left(\begin{matrix}
			u^{(0)}\\v^{(0)}
		\end{matrix}\right)\right),
	\end{equation*}
	let $G_0$ be defined by \eqref{eq:subgroup} and then let $\g_0\times E_0$ be the enhanced standard representation of $G_0$. For convenience of notations, if $G=\GL_n(\C)$, let $x_0=0$, let $G_0$ be the trivial group, let $\breve G_0=\{\pm 1\}$, and let $\g_0\times E_0$ be the zero space.
	
	Now we have
	\begin{equation*}
		\breve{G}_x\simeq\left(\breve{G}_0\right)_{x_0}\times_{\{\pm 1\}}\breve{\GL}_{n_1}(\C)_{x_1}\times_{\{\pm 1\}}\dots\times_{\{\pm 1\}}\breve{\GL}_{n_b}(\C)_{x_b}.
	\end{equation*}
	It is clear that
	\begin{equation*}
		N_{G_0x_0,x_0}^{\g_0\times E_0}\times N_{\GL_{n_1}(\C)x_1,x_1}^{\gl_{n_1}(\C)\times\C^{n_1\times 1}\times\C^{1\times n_1}}\times\dots\times N_{\GL_{n_b}(\C)x_b,x_b}^{\gl_{n_b}(\C)\times\C^{n_b\times 1}\times\C^{1\times n_b}}
	\end{equation*}
	is a subrepresentation of $N_{O,x}^{\g\times E}$. By \eqref{eq:descent1} and comparing dimensions, we obtain
	\begin{equation}
		\label{eq:descent2}
		N_{O,x}^{\g\times E}\simeq N_{G_0x_0,x_0}^{\g_0\times E_0}\times N_{\GL_{n_1}(\C)x_1,x_1}^{\gl_{n_1}(\C)\times\C^{n_1\times 1}\times\C^{1\times n_1}}\times\dots\times N_{\GL_{n_b}(\C)x_b,x_b}^{\gl_{n_b}(\C)\times\C^{n_b\times 1}\times\C^{1\times n_b}}
	\end{equation}
	
    \begin{proof}[Proof of Theorem \ref{thm:mainmvwsimple}]
        The theorem follows from Propositions \ref{prop:descentglnmvw}, \ref{prop:descentmvw} and \eqref{eq:descent2}.
    \end{proof}
	
\end{document}